\documentclass[12pt, a4paper]{amsart}

\usepackage{ifthen}
\usepackage{tikz}

\usepackage{dsfont} %pacchetto per usare il carattere della funzione caratteristica

\newtheorem*{rep@theorem}{\rep@title}
\newcommand{\newreptheorem}[2]{%
\newenvironment{rep#1}[1]{%
 \def\rep@title{#2 \ref{##1}}%
 \begin{rep@theorem}}%
 {\end{rep@theorem}}}
\makeatother

\usepackage{breqn}
\usepackage{amscd,amsmath,amssymb,amsthm,amsfonts}
\usepackage[alphabetic]{amsrefs}
\usepackage{mathrsfs}
\usepackage[shortlabels]{enumitem}
\usepackage[all]{xy}
\usepackage{thmtools}

\usepackage{xcolor}
\usepackage[hypertexnames=true,colorlinks = true, allcolors=blue,linktoc = all, pdffitwindow = false, urlbordercolor = white]{hyperref}%
\usepackage{tikz}
 
\usetikzlibrary{knots}% Added
\usetikzlibrary{decorations.markings}% Added
\usepackage{hyperref}
\usepackage{graphicx} %%ADDED last
\usepackage{neuralnetwork}

\def\ker{\operatorname{ker}}

\def\max{\operatorname{max}}

 \newcommand{\IR}[0]{\mathbb{R}}

\newcommand{\TL}[2]{%
\vcenter{\hbox{\begin{tikzpicture}
\foreach \x/\y/\z/\w in {#2} {
  \ifthenelse{\x = \z \AND \y = \w}
  {}
  {
    \ifthenelse{\y = \w}
    {\ifthenelse{\y = 0}
      {\draw (\x, \y) .. controls +(0, 0.5) and +(0, 0.5) .. (\z, \w);}
      {\draw (\x, \y) .. controls +(0, -0.5) and +(0, -0.5) .. (\z, \w);}
    }
    {\draw (\x, \y) .. controls +(0, 1) and +(0, -1) .. (\z, \w);}
  }
}
\foreach \x/\y/\z/\w in {#2} {
  \fill (\x, \y) circle (2pt); % Draw a circle at the starting point
  \fill (\z, \w) circle (2pt); % Draw a circle at the ending point
}
\draw (-0.5, 0) rectangle ({#1 - 0.5}, 1);
\clip (-0.6, 0) rectangle ({#1 - 0.6}, 1);
\end{tikzpicture}
}}}

\newreptheorem{theorem}{Theorem}%NEW
\newtheorem{theorem}{Theorem}[section]
\newtheorem*{theorem*}{Theorem}%ADDED
\newtheorem*{proposition*}{Proposition}%ADDED

\newtheorem{lemma}[theorem]{Lemma}
\newtheorem*{lemma*}{Lemma}%ADDED
\newtheorem{example}[theorem]{Example}

\newtheorem{corollary}[theorem]{Corollary}
\newtheorem{remark}[theorem]{Remark}

\numberwithin{equation}{section}

\begin{document}

%%%%% 
\allowdisplaybreaks  % Permette di spezzare le equazioni nelle pagine

\title{Piecewise Linear Equivariant Maps for Compact Groups}
\author{Valeriano Aiello} 
\address{Valeriano Aiello, Dipartimento di Matematica, Universit\`a di Roma Sapienza, P.le Aldo Moro 5, 00185 Roma, Italy, \url{https://github.com/valerianoaiello}}
\email{valerianoaiello@gmail.com} 
\begin{abstract}
Motivated by equivariant neural networks, we study piecewise linear
equivariant maps between finite-dimensional real representations of compact
groups.
We show that all genuinely non-linear piecewise linear behaviour is
confined to the subspaces on which the identity component of the group acts
trivially, while equivariance forces linearity on the corresponding
orthogonal complements.
As a consequence, we obtain a compact-group analogue of the finite-group
existence criterion of Gibson--Tubbenhauer--Williamson for non-zero
equivariant piecewise linear maps between irreducible representations,
with the identity component giving rise to a rigidity phenomenon absent
from the finite-group case.
\end{abstract}

\maketitle

\section*{Introduction}
Ordinary feed-forward neural networks are %usually 
obtained by composing affine linear maps with a specified non-linear function
$\sigma:\mathbb{R}\to\mathbb{R}$, called an activation function, which is
applied componentwise.
A particularly relevant class of non-linearities is provided by piecewise
linear maps. The most familiar example is the rectified linear unit
$\operatorname{ReLU}(x):=\max\{0,x\}$, $x\in\mathbb R$. When ReLU is
applied componentwise, the resulting map is piecewise linear, and
compositions of affine and piecewise linear maps remain piecewise linear.

In their simplest form, equivariant neural networks 
can be described as compositions of equivariant linear maps between
representations of a group, alternated with suitable equivariant non-linear
maps. Thus, the vector spaces constituting the layers of ordinary
feed-forward neural networks are endowed with group representations, while
the componentwise application of the activation function is replaced by
%suitable 
non-linear equivariant maps. In particular, when the group action is
compatible with a distinguished basis, as for permutation representations,
such maps may still be obtained by applying an activation function
componentwise, but for general representations this need not be possible.

It is therefore important to determine what classes of equivariant
non-linear maps may occur and, given that piecewise linear maps are common
in feed-forward neural networks, they constitute a natural class to begin
with and are the focus of the present article.

The study of the function spaces represented by neural networks has a long
history. Classical universal approximation results show that certain
classes of neural networks are dense in spaces of continuous functions.
A classical universal approximation result goes back to the work of Cybenko
\cite{Cybenko89}, who proved that finite linear combinations of functions
obtained from a sigmoidal activation function can approximate, with respect
to the uniform norm, any continuous scalar-valued function on the unit cube
in $\mathbb{R}^n$.
Later, this result was extended by Hornik to a broader class of activation
functions \cite{Hornik} and, finally, it was shown in \cite{LLPS} that,
under suitable assumptions, non-polynomiality is the key ingredient for all
these density results.
The study of the function spaces represented by neural networks, including
deep architectures, is still an active topic; see, e.g., \cite{BDVRV}, or
\cite{Grigsby25}, where a notion of functional dimension for feed-forward neural
networks was introduced.

The compatibility between pointwise activations, choices of coordinates,
and finite-dimensional group representations has recently been studied
systematically by Pacini--Dong--Lepri--Santin \cite{Pacini24}, who obtain,
in particular, strong restrictions for connected compact groups. Their
setting is complementary to ours: pointwise activations form a special,
basis-dependent class of non-linear maps, whereas we consider arbitrary
piecewise linear equivariant maps between finite-dimensional representations.

 In the framework of finite groups, a piecewise linear
representation-theoretic approach was pursued by
Gibson--Tubbenhauer--Williamson in \cite{Gibson24}. Their neural-network
framework uses layers that are direct sums of permutation representations
(i.e., representations of the form $\operatorname{Fun}(X,\mathbb R)$,
where $X$ is a finite $G$-set and $G$ acts by
$g\cdot f(x):=f(g^{-1}x)$).
One of their main results is the following: if $G$ is finite and
$L$ and $K$ are irreducible real $G$-representations, then the existence
of a non-zero piecewise linear equivariant map is completely determined by
the kernels of the two representations: such a map exists if and only if
$\ker(\rho_L)\subseteq\ker(\rho_K)$.

The aim of the present article is to investigate the corresponding problem
for compact groups. Such groups arise naturally in equivariant settings,
especially in problems involving rotational symmetries. For instance,
$SO(2)$- and $SO(3)$-equivariance appears in architectures for spherical
and three-dimensional geometric data; see, for example,
\cite{Cohen18,Thomas18}. These examples motivate the compact-group setting.
We restrict here to finite-dimensional real representations and study
piecewise linear equivariant maps from a representation-theoretic point of
view.

The finite-group kernel criterion therefore does not extend unchanged to
compact groups (see Example~\ref{example:kernel-criterion}), and indeed the
compact-group case displays a phenomenon absent from the finite-group case:
the identity component $G^\circ$ imposes a strong rigidity on piecewise
linear equivariant maps.

Let $L$ and $K$ be such representations, and choose $G$-invariant inner
products on them. We prove that the vector space
$\operatorname{Hom}^{\mathrm{pl}}_G(L,K)$ of $G$-equivariant piecewise
linear maps is isomorphic to the direct sum
$\operatorname{Hom}^{\mathrm{pl}}_G(L^{G^\circ},K^{G^\circ})
\oplus
\operatorname{Hom}_G\bigl((L^{G^\circ})^\perp,(K^{G^\circ})^\perp\bigr)$,
where $L^{G^\circ}\subset L$ and $K^{G^\circ}\subset K$ denote the
subspaces fixed by the action of $G^\circ$.
Moreover, the action of $G$ on the fixed-point spaces
$L^{G^\circ}$ and $K^{G^\circ}$ has finite image. Hence the non-linear
part is governed by a finite quotient of $G$, whereas the part on which
the identity component acts non-trivially is rigid.

For connected compact groups, all genuinely non-linear dependence is
confined to maps from the maximal trivial subrepresentation of the domain
to the maximal trivial subrepresentation of the codomain; on their
orthogonal complements, equivariance forces the map to be linear.
In particular, if $L$ and $K$ are non-trivial irreducible representations
of a connected compact group, then
$\operatorname{Hom}^{\mathrm{pl}}_G(L,K)\neq0$ if and only if
$L$ and $K$ are isomorphic.

More generally, our structure theorem yields a piecewise linear analogue
of Lemma \ref{Schur} (Schur)  for irreducible representations of arbitrary compact
groups. If $G^\circ$ acts trivially on the domain, the representation on
the domain factors through a finite quotient and the kernel criterion of
\cite{Gibson24} is recovered. If $G^\circ$ acts non-trivially, every
equivariant piecewise linear map is affine, and 
a non-zero such map exists precisely when $L\cong K$ or
$K\cong\mathbb R_{\mathrm{triv}}$.

The rigidity mechanism behind the structure theorem is geometric. If a
piecewise linear map is not globally affine, its non-affine locus contains
a relatively open subset of an affine hyperplane belonging to a finite
family determined by a polyhedral subdivision. Equivariance forces the
group to permute these hyperplanes. Since $G^\circ$ is connected, it fixes
every hyperplane in such a finite orbit, producing a non-zero
$G^\circ$-fixed vector in the dual representation. When the source has no
$G^\circ$-fixed vectors, this is impossible.

\section{Preliminaries and notation}\label{sec1}

In this section we provide the basic definitions needed for the paper and set the notations.
There are three brief subsections, one dealing with group representations, one  with piecewise linear maps, 
and one with equivariant neural networks.

\subsection{Group representations}
 $G$ denotes a compact group and $G^\circ$ is its identity
component. Throughout the paper, all representations are assumed to be real,
finite-dimensional, and continuous.
A representation of $G$ on a real vector space $L$ is a continuous group
homomorphism
\[
\rho_L\colon G\longrightarrow \operatorname{GL}(L).
\]
Two $G$-representations $L$ and $K$ are said to be equivalent if there exists a linear isomorphism
$T\colon L\to K$ such that
$T\circ\rho_L(g)=\rho_K(g)\circ T$ for every $g\in G$.
A representation is called irreducible, or simple, if its only
$G$-invariant linear subspaces are $\{0\}$ and $L$. We denote by
$\mathbb{R}_{\mathrm{triv}}$ the one-dimensional trivial real
representation of $G$, namely $\mathbb{R}$ endowed with the action
\[
\rho_{\mathrm{triv}}(g)x=x\qquad \forall g\in G, x\in\mathbb{R}.
\]
A map
$\Phi\colon L\longrightarrow K$
between two $G$-representations is called $G$-equivariant if
\[
\rho_K(g)\Phi(l)
=
\Phi\bigl(\rho_L(g)l\bigr) \qquad \forall g\in G, l\in L..
\]
When  $\Phi$ is linear, it is also called an intertwiner. We denote by
\[
\operatorname{Hom}_G(L,K)
\]
the vector space of $G$-equivariant linear maps from $L$ to $K$.

For a subgroup $H\leq G$, we write
\[
L^H
:=
\{v\in L:\rho_L(h)v=v
\text{ for every }h\in H\}
\]
for the corresponding fixed-point subspace. When the group $H$ is generated by a single element $h$, we write $L^h$ instead of $L^{\langle h\rangle}$.

We record two standard  results that we will use in several occasions. We omit their proofs.
\begin{lemma}\label{Schur}
Let $G$ be a compact group. The following statements hold 
\begin{enumerate}
\item (Schur) if $L$ and $K$ are real irreducible representations, and $\Phi: L\to K$ is a linear $G$-equivariant map, then either $\Phi$ is zero or an isomorphism.
\item  if $L$ and $K$ are real representations, and $\Phi: L\to K$ is an affine $G$-equivariant map, i.e., of the form $\Phi(l)=Al+b$,
 then   $A\in \operatorname{Hom}_G(L,K)$, $b\in K^G$.
\end{enumerate}
\end{lemma}

Since $G$ is compact, every finite-dimensional real $G$-representation
admits a $G$-invariant inner product. We fix such an inner product on each
representation throughout the paper. In particular, orthogonal complements
of $G$-invariant subspaces are again $G$-invariant.

\subsection{Piecewise linear maps}

We next recall the notion of piecewise linearity used throughout the paper.
Let $V$ be a finite-dimensional real vector space. A subset
$A\subseteq V$ is called convex polyhedral if it is the
intersection of finitely many closed affine half-spaces of $V$, \cite[\S 19]{RockafellarConvex}. 
 A polyhedral covering of a polyhedral subset $P\subseteq V$ is a finite
collection $\mathcal A=(A_1,\ldots,A_m)$
of convex polyhedral subsets such that
$P=\bigcup_{i=1}^m A_i$.

Let $V$ and $W$ be finite-dimensional real vector spaces and let
$A\subseteq V$. A map $f\colon A\longrightarrow W$ 
is called affine linear if there exist a linear map
$M\colon V\to W$ and a vector $b\in W$ such that
$f(x)=M(x)+b$ 
for every $x\in A$.

Let $P\subseteq V$ be a polyhedral subset. A map
\[
f\colon P\longrightarrow W
\]
is called piecewise linear  if there exists a polyhedral covering
$\mathcal A=(A_1,\ldots,A_m)$ of $P$ such that $f|_{A_i}$ is
affine linear for every $i$.

This convention is slightly more general than the one used in
\cite{Gibson24}, where convex polyhedral subsets are required to have
non-empty interior in the ambient vector space. Here we allow polyhedral
subsets of lower dimension as well. This is convenient, for instance, when
restricting piecewise linear maps to affine subspaces or to faces of a
polyhedral decomposition.

When the domain is the whole ambient vector space, however, the two
conventions agree. Indeed, in a finite polyhedral covering of $V$, the
full-dimensional members already cover $V$. Otherwise their union, being
closed, would have a non-empty open complement, which would have to be
covered by finitely many lower-dimensional polyhedra, an impossibility.

Piecewise linear maps are continuous. Moreover, whenever the relevant
compositions and sums are defined, piecewise linearity is preserved under
composition, finite direct sums, and finite sums.

We denote by $\operatorname{Hom}^{\mathrm{pl}}_G(L,K)$
the vector space of $G$-equivariant piecewise linear maps from $L$ to
$K$.

\subsection{Equivariant neural networks}

We briefly recall the connection with equivariant neural networks that
motivates the representation-theoretic questions considered below.

In a finite-dimensional equivariant neural-network, each layer
is a finite-dimensional $G$-representation. Linear maps between
consecutive layers are required to be $G$-equivariant, and non-linear
operations must satisfy the same compatibility with the group action.  

A finite-dimensional $G$-equivariant feed-forward architecture
may schematically be built by alternating $G$-equivariant affine
transformations with $G$-equivariant non-linear maps. More precisely and as recalled in Lemma \ref{Schur}-(2), an
affine map
\[
x\longmapsto T_i x+b_i
\]
from a $G$-representation $L_i$ to a $G$-representation $L_{i+1}$
is $G$-equivariant precisely when
$T_i\in \operatorname{Hom}_G(L_i,L_{i+1})$
%\qquad\text{and}\qquad
and
$b_i\in L_{i+1}^G$.
%\]
Thus an architecture with two hidden layers may schematically be written as
\[
L_0
\xrightarrow{x\mapsto T_0x+b_0}
L_1
\xrightarrow{\sigma_1}
L_1
\xrightarrow{x\mapsto T_1x+b_1}
L_2
\xrightarrow{\sigma_2}
L_2
\xrightarrow{x\mapsto T_2x+b_2}
L_3,
\]
where %the $T_i$ are $G$-equivariant linear maps, the biases $b_i$ are $G$-fixed vectors, and
 the $\sigma_i$ are $G$-equivariant
non-linear maps. The precise form of the non-linear maps depends strongly
on the representations being used.

A familiar class of non-linearities is obtained by applying a scalar
activation function componentwise with respect to a chosen basis. If the
activation function is piecewise linear, then the resulting componentwise
map is piecewise linear. For a general $G$-representation, however, such
a map need not be $G$-equivariant. Equivariance of a componentwise
activation therefore requires additional compatibility between the group
action, the chosen basis, and the activation function. For a systematic
study of this compatibility for point-wise activations and finite-dimensional
group representations, see \cite{Pacini24}.

The framework of  finite groups studied in
\cite{Gibson24} provides a particularly natural setting
in which this compatibility is automatic.  There the layers 
are direct sums of permutation representations, all linear maps between
layers are $G$-equivariant, and a fixed activation function is applied
componentwise in each hidden layer. With respect to the distinguished
permutation bases, the action of $G$ permutes coordinates. Consequently,
applying the same scalar activation function to every coordinate commutes
with the $G$-action. In particular, when the activation is piecewise
linear, as for the ReLU function, %considered in \eqref{ReLU}, 
the resulting
non-linear maps are piecewise linear and $G$-equivariant.

The representation-theoretic problem considered in the present article is
broader than this permutation-representation setting. We do not assume that
the representations are permutation representations, nor do we assume that
a piecewise linear equivariant map arises by applying an activation function
componentwise. Instead, we study arbitrary piecewise linear
$G$-equivariant maps
\[
\Phi\colon L\longrightarrow K
\]
between finite-dimensional real continuous representations of a compact
group $G$. The irreducible case will then arise as a consequence of the
general structure theorem. In this sense, the notion of piecewise linear
equivariance considered here is intrinsic and does not depend on the choice
of a basis.

\section{Main results}\label{sec2}

For finite groups, Gibson--Tubbenhauer--Williamson obtain a simple
criterion for the existence of non-zero piecewise linear equivariant maps
between irreducible representations. 
\begin{theorem} \cite[Theorem~2H.3]{Gibson24}
 If $G$ is finite and $L$
and $K$ are irreducible real $G$-representations, then
\[
\operatorname{Hom}^{\mathrm{pl}}_G(L,K)\neq 0
\quad\Longleftrightarrow\quad
\ker(\rho_L)\subseteq\ker(\rho_K);
\]
\end{theorem}

The proof %in the finite-group case 
of this result uses in an essential way the finiteness
of the group. After factoring out $\ker(\rho_L)$, one may assume that
the action on $L$ is faithful. Given a piecewise linear function
$f\colon L\to\mathbb R$ and a vector $k\in K$, one considers
\[
\Phi(x)
=
\sum_{h\in G}
f\bigl(\rho_L(h^{-1})x\bigr)\rho_K(h)k.
\]
The finiteness of $G$ ensures that $\Phi$ is again piecewise linear and that there exists a vector $l\in L\setminus\left( \cup_{g\in G\{1\}}L^g\right)$. Then one can choose a function $f$ such that $f(\rho_L(g)l)=\delta_{g,1}$, which suffices to prove that  
$\Phi$ is non-zero.

For a compact group it is natural to replace the finite sum by Haar
averaging,
\[
\Phi(x)
=
\int_G
f\bigl(\rho_L(h^{-1})x\bigr)\rho_K(h)k\,dm(h).
\]
There are some issues with this strategy. First, 
$L\setminus\left( \cup_{g\in G\setminus \{1\}}L^g\right)$ might be empty, think for example of $SO(3)$ acting on $\IR^3$ 
with its natural representation. Any vector $v\neq0$ is fixed by the rotations around it and thus 
$\cup_{g\in SO(3)\setminus \{1\}}(\IR^3)^g=\IR^3$. Second, 
although $\Phi$ is $G$-equivariant, Haar averaging does not,
in general, preserve piecewise linearity and to boot there might not even be any non-zero piecewise linear intertwiner at all 
as the following example shows. 
\begin{example} 
\label{example:kernel-criterion}
Let
 $G=S^1$,
 $L=K=\mathbb R^2\simeq\mathbb C$,
 and define
\[
\rho_L(e^{i\theta})z=e^{i\theta}z,
\qquad
\rho_K(e^{i\theta})z=e^{2i\theta}z.
\]
These representations are irreducible and inequivalent (the latter can be seen by looking at the spectrum of $\rho_L(-1)=-I$ and $\rho_K(-1)=I$). 
Their
kernels are
 $\ker(\rho_L)=\{1\}$ and
 $\ker(\rho_K)=\{\pm1\}$,
 so
 $\ker(\rho_L)\subseteq\ker(\rho_K)$.
  
Now let $f\in \operatorname{Hom}^{\mathrm{pl}}_G(L,K)$. Take an open neighborhood $U$ where $f$ is affine. 
From the $S^1$-equivariance, it follows that $f(e^{i\theta}l)=e^{i2\theta}f(l)$ for all $l\in L$. Note that the multiplication by $e^{i\theta}$ is a linear operation. Take $\theta$ sufficiently small so that $U\cap (\rho_L(e^{i\theta}) U)\neq \emptyset$ contains an open set, 
then $f$ is affine in $(\rho_L(e^{i\theta}) U)$ and the formula must coincide with that on $U$. 
If the map is not globally affine, there is a point $0\neq l\in L$ where the formula of $f$ changes. From the above argument 
it follows that the formula changes on the whole circle of radius $\| l\|$. A circle cannot be decomposed as a polygonal with finitely many sides, so $f$ must be globally affine: $f(l)=Al+b$, where $b$ is a fixed vector by $\rho_K$.
As there are no such vectors, this forces $b=0$ and $f$ is linear. However, $L$ and $K$ are inequivalent representations and thus $f=0$.   
\end{example} 

The main phenomenon in the compact setting is instead a rigidity imposed
by the identity component. 

\smallskip
\noindent
\textbf{Notations.}
From now on, $G$ denotes a compact group and $G^\circ$ its identity
component. All representations are assumed to be real, finite-dimensional,
continuous, and endowed with a $G$-invariant inner product. Given a representation  $\rho: G\to GL(M)$ and a sub-representation $N$, we denote by $P_N: M\to N$ the
 corresponding
 orthogonal projection.

\smallskip
We begin with a simple lemma and a technical one that we will use repeatedly.
\begin{lemma} 
\label{irreduciblecomponents}
For every $G$-representation $L$, the fixed-point space
$L^{G^\circ}$ is a $G$-subrepresentation of $L$.
\end{lemma}
\begin{proof}
The identity component $G^\circ$ is a normal subgroup of $G$. 
Take $l\in L^{G^\circ}$, $g\in G$, and $h\in G^\circ$, then $g^{-1}hg\in G^\circ$ and
$$
\rho(h) \rho(g)l=
\rho(g) (\rho(g^{-1})\rho(h) \rho(g))l=\rho(g)l
$$
which shows that $\rho(g)l\in L^{G^\circ}$.
We note that this lemma did not need the compactness hypothesis.
\end{proof}
The following is the main technical result of the article.
\begin{lemma} 
\label{lem:connected-compact-PL}
Let $L$ and $K$ be $G$-representations and assume that
\[
L^{G^\circ}=\{0\}.
\]
Then every $G$-equivariant piecewise linear map
\[
F\colon L\longrightarrow K
\]
is affine linear.

In particular, if $L$ is irreducible and $G^\circ$ acts non-trivially
on $L$, then every element of
$\operatorname{Hom}^{\mathrm{pl}}_G(L,K)$ is affine linear.
\end{lemma}
 \begin{proof}
 Suppose that $F$ is not globally affine, and let $N(F)$ denote the set of
points at which $F$ is not differentiable. Since $F$ is piecewise linear,
there exists a finite polyhedral subdivision of $L$ such that $F$ is affine
on each full-dimensional cell. Hence $N(F)$ is contained in a finite union
of affine hyperplanes.

Moreover, since $F$ is not globally affine, there exist two adjacent
full-dimensional cells on which the corresponding affine expressions are
distinct. Therefore $F$ is not differentiable on a relatively open subset
of their common codimension-one face. In particular, $N(F)$ contains a
relatively open subset of at least one affine hyperplane.

Let
\[
\mathcal H(F):=
\left\{
H\subset L :
\begin{array}{l}
H \text{ is an affine hyperplane and }\\
N(F) \text{ contains a relatively open subset of } H
\end{array}
\right\}. 
\]
Then $\mathcal H(F)$ is non-empty and finite. Indeed, if
\[
N(F)\subseteq H_1\cup\cdots\cup H_m
\]
for finitely many affine hyperplanes $H_i$, and $N(F)$ contains a
relatively open subset of an affine hyperplane $H$, 
then $H=H_i$ for some $i$, since a relatively open subset of $H$
cannot be contained in a finite union of proper affine subspaces of $H$.

Moreover, equivariance implies that
\[
\rho_L(g)N(F)=N(F)
\qquad\text{for every }g\in G,
\]
since pre- and post-composition with linear isomorphisms preserve
differentiability. Hence, if $H\in\mathcal H(F)$, then
$\rho_L(g)H\in\mathcal H(F)$. Thus $G$ acts on the finite set
$\mathcal H(F)$.

Fix $H\in\mathcal H(F)$. By continuity of the orbit map, the orbit of
$H$ under the connected group $G^\circ$ is connected and contained in
the finite set $\mathcal H(F)$; hence it consists of a single point. Therefore
\[
\rho_L(g)H=H
\qquad
\text{for every }g\in G^\circ.
\]

Write
\[
H=\{x\in L:\lambda(x)=c\},
\]
where $0\neq\lambda\in L^*$ and $c\in\mathbb R$. Recall that the dual
action is
\[
(g\cdot\lambda)(x)
=
\lambda\bigl(\rho_L(g^{-1})x\bigr).
\]
Since $\rho_L(g)H=H$, the affine hyperplanes defined by
\[
(g\cdot\lambda)(x)=c
\qquad\text{and}\qquad
\lambda(x)=c
\]
coincide. In particular, they have the same direction, and hence
\[
\ker(g\cdot\lambda)=\ker\lambda.
\] 
Thus there exists $a_g\in\mathbb R\setminus\{0\}$ such that
\[
g\cdot\lambda=a_g\lambda.
\]
The line $\mathbb R\lambda$ is therefore $G^\circ$-invariant, and the
action on it is described by a continuous character
\[
\chi\colon G^\circ\longrightarrow\mathbb R\setminus\{0\},
\qquad
g\cdot\lambda=\chi(g)\lambda.
\]

The subgroup $G^\circ$ is compact and connected, so $\chi(G^\circ)$ is a compact
connected subgroup of $\mathbb R\setminus\{0\}$ and the only such subgroup is
$\{1\}$. Hence
\[
0\neq\lambda\in(L^*)^{G^\circ}.
\]

On the other hand, a $G^\circ$-invariant inner product on $L$ gives a
$G^\circ$-equivariant Riesz isomorphism
\[
L\cong L^*.
\]
Therefore
\[
(L^*)^{G^\circ}\cong L^{G^\circ}=\{0\},
\]
a contradiction. We conclude that $F$ is globally affine linear.

For the final statement of this lemma, $L^{G^\circ}$ is $G$-invariant by Lemma \ref{irreduciblecomponents}.
If $L$ is irreducible, then
$L^{G^\circ}$ is either $0$ or $L$. The latter case is precisely the
case in which $G^\circ$ acts trivially on $L$. Hence non-triviality of
the $G^\circ$-action implies $L^{G^\circ}=0$, and the first part
applies.
\end{proof}
The previous lemma allows us to determine all piecewise  linear equivariant maps. 
The non-linearity is confined between
 the subspaces point-wise fixed by the $G^\circ$. 
 \begin{theorem}
\label{thm:PL-structure}
Let $L$ and $K$ be $G$-representations and set
\[
V:=L^{G^\circ},
\qquad
W:=V^\perp,
\qquad
U:=K^{G^\circ},
\qquad
Z:=U^\perp.
\]
Then the map
\begin{align*}
\operatorname{Hom}^{\mathrm{pl}}_G(L,K)
&\longrightarrow
\operatorname{Hom}^{\mathrm{pl}}_G(V,U)
\oplus
\operatorname{Hom}_G(W,Z),\\
F
&\longmapsto
\bigl(P_U\circ F\upharpoonright_V,\,
P_Z\circ F\upharpoonright_W\bigr)
\end{align*}
is a vector-space isomorphism.
\end{theorem}
\begin{proof} 
First of all we observe that $W^{G^\circ}=0$ and $Z^{G^\circ}=0$.
Let $v\in V$. We claim that 
the map $\varphi_v: W\to U$, defined as $\varphi_v(w):=P_U(F(v+w))$ for $w\in W$, is constant.
This map is $G^\circ$-equivariant because
 $\varphi_v(hw)=P_U(F(v+hw))=P_U(F(hv+hw))=hP_U(F(v+w))$ for all $h\in G^\circ$. 
 From Lemma \ref{lem:connected-compact-PL} 
we see that $\varphi_v$ must be affine, namely of the form   
$\varphi_v(w)=B_v w+b_v$, where $B_v\in \operatorname{Hom}_{G^\circ}(W,U)$, $b_v\in U$. 
Decompose $W$ and $U$ into irreducible sub-representations of $G^\circ$. 
Each sub-representation of $W$ is inequivalent to any sub-representation of $U$ (because $U$ is the direct sum of trivial $G^\circ$-representations, while $W$ is the direct sum of non-trivial $G^\circ$-representations).
By Lemma \ref{Schur}-(Schur), this implies that $B_v=0$. This means that the sought map $\psi$ is
\[
\psi\colon V\to U,
\qquad
\psi(v):=P_U(F(v)).
\]
Note that
\[
P_U(F(v+w))=\psi(v)
\qquad
\text{for all }v\in V,\ w\in W.
\]
The map $\psi$ is piecewise linear, being the restriction of
$P_U\circ F$ to $V$, and it is $G$-equivariant
\[
\psi(\rho_V(g)v)
=
P_U(F(\rho_L(g)v))
=
\rho_U(g)P_U(F(v))
=
\rho_U(g)\psi(v).
\]

Now we consider $P_Z\circ F$. 
The map $\psi_v: W\to Z$ defined as $W\ni w\mapsto P_Z(F(v+w))\in Z$ 
is piecewise linear because it is a composition of a linear map, $P_Z$, 
a piecewise linear map, $F$, and an affine map $w\mapsto v+w$.
It is also $G^\circ$-equivariant because
$\psi_v(\rho_L(h)w)=P_Z(F(v+\rho_L(h)w))=P_Z(F(\rho_L(h)v+\rho_L(h)w))=\rho_K(h)P_Z(F(v+w))=\rho_K(h)\psi_v(w)$
for all $h\in G^\circ$. Again by Lemma \ref{lem:connected-compact-PL}  the map
 $\psi_v$ is affine. It is actually linear because $Z^{G^\circ}=0$ and thus $\psi_v(w)=A_vw$ for all $w\in W$ and a suitable 
 $A_v\in  \operatorname{Hom}_{G^\circ}(W, Z)$. We claim that $A_v$ does not depend on $v$. Suppose by contradiction that there
 exist two vectors 
 $v_1$ and $v_2$ such that  $A_{v_1}\neq A_{v_2}$. Set
  $u=v_2-v_1$ and pick $w\in W$ so that
   $A_{v_1}w\neq A_{v_2}w$. 
   For $t\in [0,1]$, the map $t\mapsto P_Z(F(v_1+tu+w))$ is piecewise linear. This means that there exists a partition of 
   $[0,1]$ consisting of closed intervals, and our map is affine on each of them.
   After restricting to the first interval where the function is non-constant  
   we may suppose that $t\mapsto P_Z(F(v_1+tu+w))=A_{v_1+tu}(w)$ is affine.
    Up to replacing $v_1$ with $v_1+au$ and reparametrizing $t$ from $0$, we may suppose that, 
    for some $\epsilon>0$,
\[
P_Z(F(v_1+tu+w))
=
A_{v_1}(w)+tz,
\qquad t\in[0,\epsilon],
\]
for some non-zero vector $z\in Z$.

On the other hand, the map $t\mapsto P_Z(F(v_1+tu+tw))$ is piecewise linear. 
Shrinking $\epsilon>0$ if necessary, we may assume that it is affine
on $[0,\epsilon]$ as well.
For $t\in[0,\epsilon]$ we have
\[
P_Z(F(v_1+tu+tw))
=
A_{v_1+tu}(tw)
=
tA_{v_1+tu}(w)
=
tA_{v_1}(w)+t^2z.
\]
Since this map is affine on $[0,\epsilon]$, we must have $z=0$,
contradicting the choice of $z$. Hence $A_v$ is independent of $v$; denote the resulting map by
$A\colon W\to Z$. Since $F$ and $P_Z$ are $G$-equivariant, for every
$g\in G$ and $w\in W$,
\[
A(\rho_W(g)w)
=
P_Z(F(\rho_W(g)w))
=
\rho_Z(g)P_Z(F(w))
=
\rho_Z(g)A(w).
\]
Thus $A\in\operatorname{Hom}_G(W,Z)$.
Therefore
\[
F(v+w)=\psi(v)+A(w),
\qquad v\in V,\ w\in W.
\]

 We define the inverse map by associating
 to any pair 
 $(\psi, A)\in \operatorname{Hom}^{\mathrm{pl}}_G(V,U)
\oplus
\operatorname{Hom}_G(W,Z)$, 
the map 
$x\mapsto
\psi(P_Vx)+A(P_Wx)$. This is 
$G$-equivariant and piecewise linear.   
\end{proof}

\begin{remark}
\label{rem:finite-nonlinear-part}
Since $G^\circ$ acts trivially on $V$ and $U$, the non-linear term
\[
\psi\colon V\to U
\]
factors through the component group $G/G^\circ$. In fact, its relevant
action has finite image.

Indeed, let
\[
\rho_{V\oplus U}\colon G\longrightarrow \operatorname{GL}(V\oplus U)
\]
be the direct sum representation and set
\[
H:=\rho_{V\oplus U}(G).
\]
Then $H$ is a compact Lie group. By
\cite[Lemma~9.18]{HofmannMorris},
\[
H^\circ=\rho_{V\oplus U}(G^\circ)=\{e\}.
\]
Hence $H$ is discrete and compact, and therefore finite.

Thus all genuinely non-linear piecewise linear behaviour in
Theorem~\ref{thm:PL-structure} is governed by a finite quotient of $G$,
whereas on $W$ and $Z$ the map is necessarily a linear intertwiner.
\end{remark}

We now recover a sort of piecewise linear analogue of Schur's lemma for
irreducible representations, in the sense that it describes when piecewise linear maps exist.
\begin{corollary} 
\label{cor:compact-PL-Schur}
Let $L$ and $K$ be irreducible $G$-representations, and denote their
actions by $\rho_L$ and $\rho_K$. Then the following hold.

\begin{enumerate}
\item
Suppose that $G^\circ$ acts trivially on $L$. Then
\[
\operatorname{Hom}^{\mathrm{pl}}_G(L,K)\neq0
\quad\Longleftrightarrow\quad
\ker(\rho_L)\subseteq\ker(\rho_K).
\]

\item
Suppose that $G^\circ$ acts non-trivially on $L$. Then every
$G$-equivariant piecewise linear map $L\to K$ is affine linear, and
\[
\operatorname{Hom}^{\mathrm{pl}}_G(L,K)\neq0
\quad\Longleftrightarrow\quad
L\cong K
\quad\text{or}\quad
K\cong\mathbb R_{\mathrm{triv}}.
\]
\end{enumerate}
\end{corollary}

\begin{proof}
(1)
We first record a necessary condition,   showing ($\Rightarrow$). Suppose that
\[
0\neq F\in\operatorname{Hom}^{\mathrm{pl}}_G(L,K)
\]
and let
$H:=\ker(\rho_L)$.
For $h\in H$ and $x\in L$, equivariance gives
\[
\rho_K(h)F(x)
=
F(\rho_L(h)x)
=
F(x).
\]
Hence
$\operatorname{Im}(F)\subseteq K^H$.
Since $H$ is normal in $G$, the subspace $K^H$ is $G$-invariant. Indeed, if $v\in K^H$ and $g\in G$, then for every $h\in H$ we have
$g^{-1}hg\in H$ since $H\triangleleft G$. Hence
\[
\rho(h)\rho(g)v
=
\rho(g)\rho(g^{-1}hg)v
=
\rho(g)v,
\]
because $v$ is fixed by $H$. Therefore $\rho(g)v\in K^H$, so $K^H$
is $G$-invariant.
As $K$ is irreducible and $F\neq0$, we obtain
 $K^H=K$  
and
 $\ker(\rho_L)\subseteq\ker(\rho_K)$.
 
($\Leftarrow$)
Suppose now that $G^\circ$ acts trivially on $L$, that is $G^\circ\subseteq\ker(\rho_L)$.
 The image $\rho_L(G)$ is a compact subgroup of $\operatorname{GL}(L)$,
hence a compact Lie group. By
\cite[Lemma~9.18]{HofmannMorris}, applied to the surjective morphism
 $G\longrightarrow\rho_L(G)$,
 its identity component is
 $\rho_L(G)^\circ=\rho_L(G^\circ)=\{e\}$.
 Thus $\rho_L(G)$ is a compact discrete group and therefore finite.
Consequently,
\[
\overline G:=G/\ker(\rho_L)
\]
is finite.

Here we are working under the hypothesis
$\ker(\rho_L)\subseteq\ker(\rho_K)$.
 Then both representations factor through $\overline G$, and $L$ and
$K$ remain irreducible as $\overline G$-representations. Moreover,
the induced representation on $L$ is faithful, so
\[
\ker(\overline\rho_L)=\{e\}
\subseteq
\ker(\overline\rho_K).
\]
By \cite[Theorem~2H.3]{Gibson24},
\[
\operatorname{Hom}^{\mathrm{pl}}_{\overline G}(L,K)\neq0.
\] 
and the claim follows from
$\operatorname{Hom}^{\mathrm{pl}}_G(L,K)
=
\operatorname{Hom}^{\mathrm{pl}}_{\overline G}(L,K)$. 

 (2) Suppose that $G^\circ$ acts non-trivially on $L$.
Being $L^{G^\circ}$   a $G$-invariant subspace contained in an irreducible representation $L$ forces
 $L^{G^\circ}=\{0\}$.
 In the notation of Theorem~\ref{thm:PL-structure}, we therefore have
\[
V=0,
\qquad
W=L.
\]
Every $G$-equivariant piecewise linear map $F\colon L\to K$ has the
form
\[
F(x)=b+A(x),
\]
where
 $b\in K^G$
 and
 $A\colon L\to K$
 is a $G$-equivariant linear map. Thus every such map is affine.

If $A\neq0$, then $A$ is a non-zero linear intertwiner between irreducible
representations. Its kernel and image are invariant, hence $A$ is an
isomorphism and
 $L\cong K$.
 Conversely, if $L\cong K$, a $G$-equivariant isomorphism gives a
non-zero piecewise linear equivariant map.

If $b\neq0$, then $K^G\neq0$. In turn, the irreducibility implies
$K^G=K$, so $G$ acts trivially on $K$. An irreducible trivial real
representation is one-dimensional, and therefore
 $K\cong\mathbb R_{\mathrm{triv}}$.
 Conversely, when $K\cong\mathbb R_{\mathrm{triv}}$, every non-zero
constant map is $G$-equivariant and piecewise linear.
This concludes the proof of part (2).
 \end{proof}

For connected compact groups the structure theorem takes a  
simpler form.

\begin{corollary} 
\label{cor:connected-compact-PL}
Suppose that $G$ is connected. With the notation of
Theorem~\ref{thm:PL-structure}, every $G$-equivariant piecewise linear map
$F\colon L\to K$ has a unique expression
\[
F(v+w)=\psi(v)+A(w),
\qquad
v\in L^G,\quad w\in(L^G)^\perp,
\]
where
\[
A\in
\operatorname{Hom}_G\bigl((L^G)^\perp,(K^G)^\perp\bigr)
\]
and
$\psi\colon L^G\to K^G$
is an arbitrary piecewise linear map.

In particular, if $L$ and $K$ are irreducible and $L$ is non-trivial,
then every $G$-equivariant piecewise linear map $L\to K$ is affine linear,
and
\[
\operatorname{Hom}^{\mathrm{pl}}_G(L,K)\neq0
\quad\Longleftrightarrow\quad
L\cong K
\quad\text{or}\quad
K\cong\mathbb R_{\mathrm{triv}}.
\]
If both $L$ and $K$ are non-trivial, this reduces to
\[
\operatorname{Hom}^{\mathrm{pl}}_G(L,K)\neq0
\quad\Longleftrightarrow\quad
L\cong K.
\]
\end{corollary}

\begin{proof}
Since $G$ is connected,
 we have
$G^\circ=G$.
 Thus Theorem~\ref{thm:PL-structure} applies with
 $V=L^G$,
 $U=K^G$.
 The action of $G$ on $V$ and $U$ is trivial, so every piecewise
linear map
$\psi\colon V\to U$
is automatically $G$-equivariant. This gives the first assertion.

The statements for irreducible representations follow immediately from
Corollary~\ref{cor:compact-PL-Schur}.
\end{proof} 
 What follows is an immediate application of the previous results.

\begin{corollary} 
\label{cor:PL-closure}
With the notation of Theorem~\ref{thm:PL-structure}, let
$C_G(L,K)$ denote the space of continuous $G$-equivariant maps
$L\to K$, endowed with the topology of uniform convergence on compact
sets. Then, with the closure taken in $C_G(L,K)$,
\[
\overline{\operatorname{Hom}^{\mathrm{pl}}_G(L,K)}
\cong
C_G(V,U)\oplus\operatorname{Hom}_G(W,Z).
\]

In particular, if $G$ is connected, then
\[
\overline{\operatorname{Hom}^{\mathrm{pl}}_G(L,K)}
\cong
C(L^G,K^G)\oplus
\operatorname{Hom}_G\bigl((L^G)^\perp,(K^G)^\perp\bigr).
\]
\end{corollary} 
 
\section*{Acknowledgements}  
The author would like to thank Simone Del Vecchio and Stefano Rossi
(University of Bari) for helpful discussions.
 
The author is partially supported by Sapienza Universit\`a di Roma
(Progetto di Ateneo Dipartimentale 2024, "New research trends in
Mathematics at Castelnuovo").
This work was also supported by the Gruppo Nazionale per l'Analisi
Matematica, la Probabilità e le loro Applicazioni (GNAMPA--INdAM)
through the program "Partecipazione a Convegni, Scuole, Workshop e
Cicli di Seminari" and by the GNAMPA--INdAM project
“Simmetrie distribuzionali per processi stocastici quantistici” CUP E53C25002010001.

\end{document}